\documentclass{amsart}
\usepackage{amsmath,amsthm}
\usepackage{amssymb,url}

\title{Normal numbers and limit computable Cantor series}

\author{Achilles A. Beros}
\address[A. Beros]{Laboratoire d'Informatique de Nantes Atlantique\\
Universit\'e de Nantes\\
2 rue de la Houssini\`{e}re BP 92208\\
44322 Nantes Cedex 03\\
FRANCE\\}
\email{achilles.beros@univ-nantes.fr}

\author{Konstantinos A. Beros}
\address[K. Beros]{Department of Mathematics\\
University of North Texas\\
General Academics Building 435\\
1155 Union Circle, \#311430\\
Denton, TX 76203-5017}
\email{beros@unt.edu}

\theoremstyle{plain}

\newtheorem{Theorem}{Theorem}[section]

\newtheorem{Corollary}[Theorem]{Corollary}

\theoremstyle{definition}
\newtheorem{Definition}[Theorem]{Definition}

\def\upto{\upharpoonright}

\def\NN{{\mathbb N}}

\begin{document}

\begin{abstract}
Given any oracle, $A$, we construct a basic sequence $Q$, computable in the jump of $A$, such that no $A$-computable real is $Q$-distribution-normal.  A corollary to this is that there is a $\Delta^0_{n+1}$ basic sequence with respect to which no $\Delta^0_n$ real is distribution-normal.  As a special case, there is a limit computable sequence relative to which no computable real is distribution-normal.
\end{abstract}

\maketitle

\section{Introduction}

The effective theory of the reals has been an active area of research for many years.  Out of this field have come a number of effective formalizations of the intuitive concept of randomness, e.g., Martin-L\"of randomness.  There are, however, a number of classical formalizations of randomness which derive from ergodic theory.  In the present work, we explore one of these classical notions, but in an effective context.  

Given $b\in \mathbb N$, a real number $x$ is said to be {\em $b$-normal} if the numbers $x, bx, b^2x, \ldots$ are uniformly distributed modulo 1.  That is, for each interval $I \subseteq [0,1]$ of length $\varepsilon$, one has
\[
\lim_{n \rightarrow \infty} \frac{\big| \{ k < n : b^k x (\mbox{mod 1}) \in I\} \big|}{n} = \varepsilon.
\]
Historically, number theorists have developed several methods for algorithmically producing $b$-normal numbers.  One of the best known such methods is the Champernowne construction  (see \cite{champ}).  If $p_i \in b^{<\omega}$ is the base-$b$ expansion of $i \in \NN$, then the real number with $b$-ary expansion
\[
0.p_0 \, p_1 \, p_2 \ldots
\]
is $b$-normal.  For instance, the non-negative integers are 0, 1, 10, 11, 100, $\ldots$ in base 2, and the real with binary expansion
\[
0.0\, 1\, 10\, 11\, 100\, 101\, 110\, 111 \ldots
\]
is 2-normal.  In essence, the Champernowne construction shows that, for each $b$, there is a computable real number which is $b$-normal.

One may generalize the notion of $b$-ary expansions of real numbers to that of so-called ``Cantor series expansions'' (see \cite{cantor}).  Given a sequence $Q = (q_n)_{n \in \NN}$ of positive integers, with each $q_n \geq 2$, and a real number $x \in (0,1)$, there exist integers $a_0 , a_1 , \ldots$ such that $0 \leq a_n < q_n$, for each $n$, and 
\[
x = \sum_{n=0}^\infty \frac{a_n}{q_0 q_1 \ldots q_n}.
\]
This expansion is known as the {\em Cantor series expansion of $x$}, with respect to the {\em basic sequence} $Q$.  Over the years, there has been some study of Cantor series expansions under different assumptions on the basic sequence $(q_n)_{n \in \NN}$.  For instance, see \cite{erdos.renyi:1} and \cite{erdos.renyi:2}.

There is a corresponding generalization of $b$-normality in the context of Cantor series.  Specifically, if $Q = (q_n)_{n \in \NN}$ is a sequence of positive integers, with each $q_n \geq 2$, then $x \in (0,1)$ is said to be {\em $Q$-distribution-normal} if and only if the sequence $x, q_0 x , q_0 q_1 x , q_0 q_1 q_2 x, \ldots$ is uniformly distributed modulo 1.  Thus, $b$-normality is equivalent to $Q$-distribution-normality for $Q = (q_n)_{n\in \NN}$, with each $q_n = b$.

It is an active area of research in modern number theory to try to find constructions analogous to the Champernowne construction in the context of Cantor series and other expansions of real numbers (e.g., continued fractions, L\"uroth expansions, etc.).  Examples of these lines of inquiry can be found in \cite{altomare.mance}, \cite{AKS}, \cite{madritsch} and \cite{madritsch.thuswaldner}.  There has also been work on relating the various classical notions of normality with recursion theoretic and descriptive set theoretic measures of complexity and randomness.  See, for example, \cite{ki.linton}, \cite{becher.heiber.slaman.1} and \cite{becher.heiber.slaman.2}. 

In order to obtain algorithmic constructions of normal numbers in the context of Cantor series, one often places conditions on the sequence $(q_n)_{n \in \NN}$ that guarantee rapid divergence to infinity, e.g., that $\sum_n 1/q_n < \infty$.

In the present work, we provide a group of results which serve as a counterpoint to such attempts to algorithmically produce normal numbers.  The following theorem is our main result.

\begin{Theorem}\label{T:1}
There is a $\Delta_2^0$ basic sequence $Q$ (consisting of powers of 2) such that no computable real number is $Q$-distribution-normal.
\end{Theorem}

\section{Preliminaries}

As we are presenting Theorem~\ref{T:1} in the context of basic sequences consisting of powers of 2 (although it could just as easily be done with an arbitrary $b$), we introduce some notation for working with binary expansions of real numbers in $[0,1]$.

\medskip

\noindent\textsc{Notation:}
\begin{enumerate}
\item If $\alpha \in 2^\NN$, let $x_\alpha$ denote the real number $\sum_{n \in \NN} \frac{\alpha (n)}{2^{-n-1}}$.

\item If $n \in \NN$ and $\alpha \in 2^\NN$, we will write $n \alpha$ for $(\alpha(n) , \alpha (n+1), \ldots)$, i.e., $n\alpha$ is the $n$-bit left shift of $\alpha$.
\end{enumerate}

Suppose that $Q = (q_n)_{n \in \NN}$ with each $q_n = 2^{s_n}$, for some integers $s_n \geq 1$.  If $\alpha \in 2^\NN$ and $\alpha$ does not end with an infinite string of 1's, then, for each $n$ and $p = s_0 + \ldots + s_n$, we have $q_0 \cdot \ldots \cdot q_n x_{\alpha} ({\rm mod}\, 1) =  x_{p \alpha}$.

The following is our key computability-theoretic definition.

\begin{Definition}
We say that $x \in [0,1]$ is $\Delta^0_n$ if and only if there is an $\alpha \in 2^\NN$ such that $\{ n \in \omega : \alpha (n) = 1\}$ is a $\Delta^0_n$ subset of $\NN$ and $x = x_\alpha$.
\end{Definition}

Recall that a subset $A\subseteq \NN$ is $\Delta^0_n$ if and only if $A$ is computable in $0^{(n)}$ (the $n$-fold jump of $\emptyset$).  Our definition of $\Delta^0_n$ for $x \in [0,1]$ is equivalent to the standard definition of $\Delta^0_n$ for the associated real $\psi (\alpha) \in [0,1]$ (see \cite[\S 1.8]{nies}).

Next, we require an enumeration of all computable reals.  Note that an enumeration of all computable reals will include c.e.~reals as well, unless an appropriate oracle is introduced.  To avoid the extra complexity inherent in dealing with partial functions, we define a slightly modified universal Turing machine.

\begin{Definition}
Let $\{\phi_{e,s}\}_{e,s\in\mathbb N}$ be the standard enumeration of all binary-valued partial computable functions.  We define an array of computable functions, $\{\phi^*_{e,s}\}_{e,s\in\mathbb N}$, as follows:
\[
\phi^*_{e,s}(x) = \begin{cases}
\phi_{e,s}(x) & \mbox{ if } \phi_{e,s}(x)\downarrow \\
0 & \mbox{ otherwise. }
\end{cases},
\]
\end{Definition}

Unlike the standard universal Turing machine, $\phi^*$ may change its values.  Each value, however, will change at most once and from $0$ to $1$, if it does change.  The sequence, $\{\phi_e^*\}_{e\in \mathbb N}$, serves as an enumeration of the computable reals, although it is obviously not a computable enumeration.  We will freely identify each $\phi^*_{e,s}$ with the infinite sequence it codes.

Following the notation introduced above, we let $n \phi^*_{e,s}$ denote the $n$-bit left shift of the infinite sequence determined by $\phi^*_{e,s}$, i.e., if $\phi^*_{e,s}$ codes the sequence $\alpha$, then $n \phi^*_{e,s}$ codes the sequence $(\alpha(n) , \alpha (n+1) , \ldots) \in 2^\NN$.

Note that the computable reals in $[0,1]$ are exactly the reals of the form $x_{\phi^*_e}$.

\section{Diagonalizing against all computable and c.e.~reals}

To prove Theorem~\ref{T:1}, we will construct a strictly increasing $\Delta^0_2$ function $f : \NN \rightarrow \NN$ such that $Q = (q_n)_{n \in \NN}$, with $q_n = 2^{f(n+1) - f(n)}$, is a basic sequence with the property that no computable real is $Q$-distribution-normal.

As the desired function is to be $\Delta_2^0$, we will construct it as the limit of a computable sequence of finite partial functions, $\{f_s\}_{s\in \mathbb N}$.  For an arbitrary $s$, the function $f_s$ is constructed in $s+1$ stages.  We present the construction of $f_s$.

\bigskip

\noindent\textsc{Stage $0$:}  We define $f_s(0) = 0$ and end the stage.  The domain of $f_s$ is currently $[0,1) = [0,3^0)$.\par

\smallskip

\noindent\textsc{Stage $t+1$:}  We define
\[
A_k = \{ p \in (f_s(3^t - 1),\infty) : \phi^*_{t,s+1}(p) = k \}.
\]
Either $|A_0| \geq 2(3^{t})$ or $|A_1| \geq 2(3^{t})$, so let $k$ be the least of $0$ and $1$ such that $|A_k| \geq 2(3^{t})$.  We choose $p_1 < \ldots < p_{2(3^{t})}$ in $A_k$, with each $p_i$ as small as possible.  Set $f_s (3^{t}+i) = p_{i+1}$ for $i \leq 2(3^{t})-1$ and end the stage.  The domain of $f_s$ is currently $[0,3^{t+1})$.\par

\bigskip

By the pigeonhole principle, the interval $(f_s(3^t - 1),4(3^t) + f_s(3^t - 1))$ must either contain at least $2(3^t)$-many $p$ such that $\phi^*_{t+1,s+1}(p) = 0$ or $2(3^t)$-many $p$ such that $\phi^*_{t+1,s+1}(p) = 1$.  It follows that 
\[
f_s (3^{t+1} -1) \leq 4(3^t) + f_s(3^t - 1),
\]
for each $t \leq s$.  Hence, 
\begin{equation}\label{ceE1}
f_s (3^{t+1}) \leq 0 + 4 + 12 + \ldots + 4(3^t) = 2(3^{t+1} - 1)
\end{equation}
for all $s$ and $t$, with $t \leq s$.  Note that this upper bound is independent of $s$.

\bigskip

Now that we have defined $f_s$ for $s \in \mathbb N$, we define $f(x) = \lim_{s\rightarrow \infty} f_s(x)$.  To verify that we have constructed a function with the desired properties, we must prove two claims.  First, we must prove that $f$ is well-defined; in other words, for every $p \in \mathbb N$, there exists $m \in \mathbb N$ such that for all $s \geq m$, $f_s(p) = f_m(p)$.  We fix $p \in \mathbb N$ and suppose $i \in \NN$ is such that $p < 3^i$.  Pick $m\in \mathbb N$ such that if $s \geq m$, then
\[
\phi^*_{e}\upto \max\{f_a(3^i) : a \in \NN \} = \phi^*_{e,s}\upto \max\{f_a(3^i) : a \in \NN \},
\]
for all $e \leq i$.  Note that the maxima above are finite by \eqref{ceE1}.  Clearly $f_s(p) = f_m(p)$ for all $s \geq m$, since $f_s (p)$ depends only on the values of $\phi_e^*(\ell)$, for $e \leq i$ and
\[
\ell \leq \max\{f_a(3^i) : a \in \NN \} < \infty.
\]
Thus, $f$ is well-defined and therefore, $\Delta_2^0$.\par

Let $q_n = 2^{f(n+1) - f(n)}$ and let $Q = (q_n)_{n \in \NN}$.  The second claim we must verify is that no real number of the form $x_{\phi^*_e}$ is $Q$-distribution-normal.  Fix $\alpha = \phi^*_e$ and let $i_0<i_1<i_2 \ldots$ be a sequence of natural numbers such that $\phi^*_{i_k} = \alpha$ for all $k \in \mathbb N$.  We consider a single value of $k$.  From the definition of $f_s$ it is clear that either
\[
\frac{\big|\{  p \leq 3^{i_k} : x_{f(p)\alpha}\leq \frac{1}{2}  \}\big|}{3^{i_k}} \geq 2/3 \mbox{\hspace{1em} or \hspace{1em}} \frac{\big|\{  p \leq 3^{i_k} : x_{f(p)\alpha}\geq \frac{1}{2}  \}\big|}{3^{i_k}} \geq 2/3.
\]
Since this is true for all $k\in \mathbb N$ and $\phi^*_{i_k} = \phi^*_{e}$, we conclude that
\[
\lim_{n\rightarrow \infty} \frac{|\{  p \leq n : x_{f(p)\alpha}\leq \frac{1}{2}  \}|}{n}
\]
either does not exist or is not $\frac{1}{2}$.  Hence $\alpha = x_{\phi^*_{e}}$ is not $Q$-distribution-normal.  As every computable real occurs in the sequence $\{ x_{\phi^*_e} \}_{e \in \NN}$, we have proved the desired result.

\section{Generalizations}

Relativizing the proof of Theorem~\ref{T:1} to an arbitrary oracle, we obtain the following theorem.

\begin{Theorem}\label{lcA-for-cA}
Let $A$ be any subset of the natural numbers.  There is a basic sequence $Q$, limit computable in $A$, such that no computable real is $Q$-distribution-normal.
\end{Theorem}

By the relativized limit lemma, a set is limit computable in $A$ if and only if it is computable in $A'$, the jump of $A$.  As a consequence, we obtain a direct generalization of Theorem \ref{T:1} for all the ``$\Delta$-classes'' of the arithmetical hierarchy.

\begin{Corollary}
There is a $\Delta^0_{n+1}$ basic sequence $Q$ such that no $\Delta^0_n$ real is $Q$-distribution-normal.
\end{Corollary}

\begin{proof}
Setting $A = 0^{(n)}$, Theorem \ref{lcA-for-cA} guarantees the existence of a basic sequence $Q$ which is limit computable in $0^{(n)}$ and such that no real computable in $0^{(n)}$ is $Q$-distribution-normal.  If $Q$ is such a sequence, then $Q$ is computable in $0^{(n+1)}$.  Equivalently, $Q$ is $\Delta_{n+1}^0$.
\end{proof}

\bibliographystyle{plain}

\end{document}